\documentclass[reqno]{amsart}
\usepackage{amsmath, amsthm, amssymb, amstext}

\usepackage{hyperref,xcolor}
\hypersetup{
 pdfborder={0 0 0},
 colorlinks,
}
\usepackage{enumitem}
\newtheorem{theorem}{Theorem}

\newtheorem{proposition}[theorem]{Proposition}

\newtheorem{definition}[theorem]{Definition}
\newtheorem{example}[theorem]{Example}

\DeclareMathOperator*{\divergenz}{div}              %
\DeclareMathOperator*{\Ss}{S}

\newcommand{\R}{\mathbb{R}}

\newcommand{\Lp}[1]{L^{#1}(\Omega)}

\newcommand{\Wp}[1]{W^{1,#1}(\Omega)}

\newcommand{\Wpzero}[1]{W^{1,#1}_0(\Omega)}

\newcommand{\ran}{\rangle}

\newcommand{\ph}{\varphi}

\newcommand{\into}{\int_{\Omega}}

\newcommand{\weak}{\rightharpoonup}

\newcommand{\close}{\overline{\Omega}}

\numberwithin{theorem}{section}
\numberwithin{equation}{section}

\title[Double phase problems with convection term]{Existence and uniqueness results for double phase problems with convection term}

\author[L.\,Gasi\'nski]{Leszek Gasi\'nski}
\address[L.\,Gasi\'nski]{Pedagogical University of Cracow, Department of Mathematics, Podchorazych 2, 30-084 Cracow, Poland}
\email{leszek.gasinski@up.krakow.pl}

\author[P.\,Winkert]{Patrick Winkert}
\address[P.\,Winkert]{Technische Universit\"{a}t Berlin, Institut f\"{u}r Mathematik, Stra\ss e des 17.\,Juni 136, 10623 Berlin, Germany}
\email{winkert@math.tu-berlin.de}


\begin{document}

\begin{abstract}
    In this paper we consider quasilinear elliptic equations with double phase phenomena and a reaction term depending on the gradient. Under quite general assumptions on the convection term we prove the existence of a weak solution by applying the theory of pseudomonotone operators. Imposing some linear conditions on the gradient variable the uniqueness of the solution is obtained.
\end{abstract}
	
\maketitle
	
\section{Introduction}

Given a bounded domain $\Omega \subseteq \R^N$, $N\geq 2$, with Lipschitz boundary $\partial \Omega$ we consider the following double phase problem with convection term
\begin{equation}\label{problem}
    \begin{aligned}
	-\divergenz\left(|\nabla u|^{p-2}\nabla u+\mu(x) |\nabla u|^{q-2}\nabla u\right) & =f(x,u,\nabla u)\quad && \text{in } \Omega,\\
	u & = 0 &&\text{on } \partial \Omega,
    \end{aligned}
\end{equation}
where $1<p<q<N$, the function $\mu \colon \close \to [0,\infty)$ is supposed to be Lipschitz continuous and $f\colon\Omega\times\R\times\R^N\to \R$ is a Carath\'{e}odory function, that is, $x\mapsto f(x,s,\xi)$ is measurable for all $(s,\xi)\in\R\times\R^N$ and $(s,\xi)\mapsto f(x,s,\xi)$ is continuous for a.a.\,$x\in\Omega$, which satisfies appropriate conditions; see hypotheses (H) and (U1), (U2) in Section \ref{section_3}.

The main objective of this paper is to present existence and uniqueness results for problems of type \eqref{problem}. The main novelty of our paper is the combination of both, a double phase operator and a convection term which we describe below. To the best of our knowledge, this is the first paper dealing with both notions.

To be more precise, problem \eqref{problem} combines two interesting phenomena. The first one is the fact that the operator involved in \eqref{problem} is the so-called double phase operator whose behavior switches between two different elliptic situations. In other words, it depends on the values of the function $\mu\colon\close\to [0,\infty)$. Indeed, on the set $\{x\in \Omega: \mu(x)=0\}$ the operator will be controlled by the gradient of order $p$ and in the case $\{x\in \Omega: \mu(x) \neq 0\}$ it is the gradient of order $q$. This is the reason why it is called double phase. Originally, the idea to treat such operators comes from Zhikov \cite{Zhikov-1986}, \cite{Zhikov-1995}, \cite{Zhikov-1997} who introduced such classes to provide models
of strongly anisotropic materials; see also the monograph of Zhikov-Kozlov-Oleinik \cite{Zhikov-Kozlov-Oleinik-1994}. In order to describe this phenomenon, he introduced the functional
\begin{align}\label{integral_minimizer}
    \omega \mapsto \int \left(|\nabla \omega|^p+\mu(x)|\nabla \omega|^q\right)\,dx,
\end{align}
which was intensively studied in the last years. We refer to the papers of Baroni-Colombo-Mingione \cite{Baroni-Colombo-Mingione-2015}, \cite{Baroni-Colombo-Mingione-2016}, \cite{Baroni-Colombo-Mingione-2018}, Baroni-Kussi-Mingione \cite{Baroni-Kussi-Mingione-2015}, Colombo-Mingione \cite{Colombo-Mingione-2015a}, \cite{Colombo-Mingione-2015b} and the references therein concerning the regularity. We also point out that the integrals of the from \eqref{integral_minimizer} arise in the context of functionals with non-standard growth; see the works of Cupini-Marcellini-Mascolo \cite{Cupini-Marcellini-Mascolo-2015} and Marcellini \cite{Marcellini-1989}, \cite{Marcellini-1991}.

Recently, Perera-Squassina \cite{Perera-Squassina-2019} studied double phase problems and stated an existence result which was proved via Morse theory in terms of critical groups. The corresponding eigenvalue problem of the double phase operator with Dirichlet boundary condition was analyzed by Colasuonno-Squassina \cite{Colasuonno-Squassina-2016} who proved the existence and properties of related variational eigenvalues. By applying variational methods, Liu-Dai \cite{Liu-Dai-2018} treated double phase problems and proved existence and multiplicity results.

The second interesting phenomena in our work is the appearance of a nonlinearity on the right-hand side which also depends on the gradient of the solution. Such functions are usually called convection terms. The difficulty with the gradient dependent term is the nonvariational character of the problem. Nevertheless there exists a number of papers concerning existence and multiplicity results. Our starting point in this paper is the work of Averna-Motreanu-Tornatore \cite{Averna-Motreanu-Tornatore-2016} who considered problem \eqref{problem} with a homogeneous Dirichlet boundary condition and the $(p,q)$-Laplacian as defined in \eqref{operator_representation3}; see Section \ref{section_2}.

For other existence results on quasilinear equations with convection term and the $p$-Laplace or the $(p,q)$-Laplace differential operator we refer to the papers of Bai-Gasi\'nski-Papageorgiou \cite{Bai-Gasinski-Papageorgiou-2018}, De Figueiredo-Girardi-Matzeu \cite{De-Figueiredo-Girardi-Matzeu-2004}, Dupaigne-Ghergu-R\u{a}dulescu \cite{Dupaigne-Ghergu-Radulescu-2007}, Faraci-Motreanu-Puglisi \cite{Faraci-Motreanu-Puglisi-2015}, Faria-Miyagaki-Motreanu \cite{Faria-Miyagaki-Motreanu-2014a}, Faria-Miyagaki-Motreanu-Tanaka \cite{Faria-Miyagaki-Motreanu-Tanaka-2014b}, Gasi\'nski-Papageorgiou \cite{Gasinski-Papageorgiou-2017}, Marano-Winkert \cite{Marano-Winkert-2019}, Motreanu-Mo\-tre\-anu-Moussaoui \cite{Motreanu-Motreanu-Moussaoui-2018}, Motreanu-Tanaka \cite{Motreanu-Tanaka-2017}, Motreanu-Tornatore \cite{Motreanu-Tornatore-2017}, Motreanu-Winkert \cite{Motreanu-Winkert-2019}, Ruiz \cite{Ruiz-2004}, Tanaka \cite{Tanaka-2013} and the references therein.

The main idea in the existence theorem of our paper is the usage of the surjectivity result for pseudomonotone operators. This can be realized by an easy condition on the convection term, in addition to the usual growth condition. For the uniqueness result we need much stronger assumptions and narrow down to the most interesting case $2=p<q<N$ which is the generalization of the $(q,2)$-Laplace differential operator; see \eqref{operator_representation3} for its definition.

\section{Preliminaries}\label{section_2}

For $1\leq r<\infty$, we denote by $\Lp{r}$ and $L^r(\Omega;\R^N)$ the usual Lebesgue spaces equipped with the norm $\|\cdot\|_r$. If $1<r<\infty$, then $\Wp{r}$ and $\Wpzero{r}$ stand for the Sobolev spaces endowed with the norms $\|\cdot\|_{1,r}$ and $\|\cdot\|_{1,r,0}$, respectively.

Let $\mathcal{H}\colon \Omega \times [0,\infty)\to [0,\infty)$ be the function
\begin{align*}
    (x,t)\mapsto t^p+\mu(x)t^q
\end{align*}
where $1<p<q<N$ and
\begin{align}\label{condition_poincare}
    \frac{q}{p}<1+\frac{1}{N}, \qquad \mu\colon \close\to [0,\infty) \text{ is Lipschitz continuous}.
\end{align}
The Musielak-Orlicz space $L^\mathcal{H}(\Omega)$ is defined by
\begin{align*}
    L^\mathcal{H}(\Omega)=\left \{u ~ \Big | ~ u: \Omega \to \R \text{ is measurable and } \rho_{\mathcal{H}}(u):=\into \mathcal{H}(x,|u|)\,dx< +\infty \right \}.
\end{align*}
While equipped with the Luxemburg norm
\begin{align*}
    \|u\|_{\mathcal{H}} = \inf \left \{ \tau >0 : \rho_{\mathcal{H}}\left(\frac{u}{\tau}\right) \leq 1  \right \},
\end{align*}
the space $L^\mathcal{H}(\Omega)$ becomes uniformly convex and so a reflexive Banach space. Furthermore we define
\begin{align*}
    L^q_\mu(\Omega)=\left \{u ~ \Big | ~ u: \Omega \to \R \text{ is measurable and } \into \mu(x) | u|^q \,dx< +\infty \right \}
\end{align*}
and endow it with the seminorm
\begin{align*}
    \|u\|_{q,\mu} = \left(\into \mu(x) |u|^q \,dx \right)^{\frac{1}{q}}.
\end{align*}
From Colasuonno-Squassina \cite[Proposition 2.15]{Colasuonno-Squassina-2016} we have the continuous embeddings
\begin{align*}
    \Lp{q} \hookrightarrow L^\mathcal{H}(\Omega) \hookrightarrow \Lp{p} \cap L^q_\mu(\Omega).
\end{align*}
For $u\neq 0$ we see that $\rho_\mathcal{H}(\frac{u}{\|u\|_\mathcal{H}})=1$ and so, it follows that
\begin{align}\label{inequality_lp}
    \min \left\{\|u\|_\mathcal{H}^p,\|u\|_\mathcal{H}^q \right\} \leq \|u\|_p^p+\|u\|^q_{q,\mu}
    \leq \max\left \{\|u\|_\mathcal{H}^p,\|u\|_\mathcal{H}^q\right\}
\end{align}
for all $u\in L^\mathcal{H}(\Omega)$. The corresponding Sobolev space $W^{1,\mathcal{H}}(\Omega)$ is defined by
\begin{align*}
    W^{1,\mathcal{H}}(\Omega)= \left \{u \in L^\mathcal{H}(\Omega) : |\nabla u| \in L^{\mathcal{H}}(\Omega) \right\}
\end{align*}
with the norm
\begin{align*}
    \|u\|_{1,\mathcal{H}}= \|\nabla u \|_{\mathcal{H}}+\|u\|_{\mathcal{H}}
\end{align*}
where $\|\nabla u\|_\mathcal{H}=\||\nabla u|\|_{\mathcal{H}}$.

By $W^{1,\mathcal{H}}_0(\Omega)$ we denote the completion of $C^\infty_0(\Omega)$ in $W^{1,\mathcal{H}}(\Omega)$ and thanks to \eqref{condition_poincare} we have an equivalent norm on $W^{1,\mathcal{H}}_0(\Omega)$ given by
\begin{align*}
    \|u\|_{1,\mathcal{H},0}=\|\nabla u\|_{\mathcal{H}},
\end{align*}
see Colasuonno-Squassina \cite[Proposition 2.18]{Colasuonno-Squassina-2016}. Both spaces $W^{1,\mathcal{H}}(\Omega)$ and $W^{1,\mathcal{H}}_0(\Omega)$ are uniformly convex and so, reflexive Banach spaces. In addition it is known that the embedding
\begin{align}\label{embedding_sobolev}
    W^{1,\mathcal{H}}_0(\Omega) \hookrightarrow \Lp{r}
\end{align}
is compact whenever $r<p^*$, see Colasuonno-Squassina \cite[Proposition 2.15]{Colasuonno-Squassina-2016}, with $p^*$ being the critical exponent to $p$ given by
\begin{align}\label{critical_exponent}
    p^*:=\frac{Np}{N-p},
\end{align}
recall that $1<p<N$.
From \eqref{inequality_lp} we directly obtain that
\begin{align}\label{inequality_w1p}
    \min \left\{\|u\|_{1,\mathcal{H},0}^p,\|u\|_{1,\mathcal{H},0}^q \right\} \leq \|u\|_p^p+\|u\|^q_{q,\mu}
    \leq \max\left \{\|u\|_{1,\mathcal{H},0}^p,\|u\|_{1,\mathcal{H},0}^q\right\}
\end{align}
for all $u\in W^{1,\mathcal{H}}_0(\Omega)$.

Consider the eigenvalue problem for the $r$-Laplacian with homogeneous Dirichlet boundary condition and $1<r<\infty$ defined by
\begin{equation}\label{eigenvalue_problem}
    \begin{aligned}
	-\Delta_r u& =\lambda|u|^{r-2}u\quad && \text{in } \Omega,\\
	u
       & = 0  &&\text{on } \partial \Omega.
    \end{aligned}
\end{equation}

It is known that the first eigenvalue $\lambda_{1,r}$ of \eqref{eigenvalue_problem} is positive, simple, and isolated. Moreover, it can be variationally characterized through
\begin{align}\label{lambda_one}
    \lambda_{1,r} = \inf_{u \in W^{1,r}(\Omega)} \left \{\int_\Omega |\nabla u|^r \,dx : \int_\Omega |u|^r \,dx=1 \right \},
\end{align}
see L{\^e} \cite{Le-2006}.
Let us recall some definitions which we will use later.
\begin{definition}\label{SplusPM}
    Let $X$ be a reflexive Banach space, $X^*$ its dual space and denote by $\langle \cdot,\cdot\rangle$ its duality pairing. Let $A\colon X\to X^*$, then $A$ is called
    \begin{enumerate}[leftmargin=1cm]
	\item[(a)]
	    to satisfy the $(\Ss_+$)-property if $u_n \weak u$ in $X$ and $\limsup_{n\to \infty} \langle Au_n,u_n-u\rangle \leq 0$ imply $u_n\to u$ in $X$;
	 \item[(b)]
	    pseudomonotone if $u_n \weak u$ in $X$ and $\limsup_{n\to \infty} \langle A(u_n),u_n-u\rangle \leq 0$ imply $Au_n \weak u$ and $\langle Au_n,u_n\rangle \to \langle Au,u\rangle$.
    \end{enumerate}
\end{definition}

Our existence result is based on the following surjectivity result for pseudomonotone operators, see, e.\,g.\,Carl-Le-Motreanu \cite[Theorem 2.99]{Carl-Le-Motreanu-2007}.

\begin{theorem}\label{theorem_pseudomonotone}
    Let
$X$ be a real, reflexive Banach space, and let $A\colon X\to X^*$ be a pseudomonotone,
bounded, and coercive operator, and $b\in X^*$. Then a solution of the equation
$Au=b$ exists.
\end{theorem}

Let $A\colon \Wpzero{\mathcal{H}}\to \Wpzero{\mathcal{H}}^*$ be the operator defined by
\begin{align}\label{operator_representation}
    \langle A(u),\ph\rangle_{\mathcal{H}} :=\into \left(|\nabla u|^{p-2}\nabla u+\mu(x)|\nabla u|^{q-2}\nabla u \right)\cdot\nabla\ph \,dx,
\end{align}
where $\langle \cdot,\cdot\rangle_{\mathcal{H}}$ is the duality pairing between $\Wpzero{\mathcal{H}}$ and its dual space $\Wpzero{\mathcal{H}}^*$. The properties of the operator $A\colon \Wpzero{\mathcal{H}}\to \Wpzero{\mathcal{H}}^*$ are summarized in the following proposition, see Liu-Dai \cite{Liu-Dai-2018}.

\begin{proposition}\label{properties_operator_double_phase}
    The operator $A$ defined by \eqref{operator_representation} is bounded, continuous, monotone (hence maximal monotone) and of type $(\Ss_+)$.
\end{proposition}

A special case of the operator $A$ defined by \eqref{operator_representation} occurs when $\mu \equiv 0$. This leads to the operator $A_p \colon \Wpzero{p} \to \Wpzero{p}^*$ defined by
\begin{align*}
    \left \langle A_p(u),\ph \right \ran_p:= \into |\nabla u|^{p-2}\nabla u\cdot \nabla \ph\, dx,
\end{align*}
where $\langle \cdot,\cdot\rangle_p$ is the duality pairing between $\Wpzero{p}$ and its dual space $\Wpzero{p}^*$. This operator is the well-known $p$-Laplace differential operator.

Another special case happens when $\mu\equiv 1$, that is, $A_{q,p} \colon \Wpzero{q} \to \Wpzero{q}^*$ defined by
\begin{align}\label{operator_representation3}
    \left \langle A_{q,p}(u),\ph \right \ran_{qp}:= \into |\nabla u|^{p-2}\nabla u\cdot \nabla \ph \,dx+\into |\nabla u|^{q-2}\nabla u \cdot \nabla \ph \,dx,
\end{align}
where $\langle \cdot,\cdot\rangle_{qp}$ stands for the duality pairing between $\Wpzero{q}$ and its dual space $\Wpzero{q}^*$, is the  so-called $(q,p)$-Laplace differential operator.

\section{Main results}\label{section_3}

We assume the following hypotheses on the right-hand side nonlinearity $f$.

\begin{enumerate}[leftmargin=1cm]
    \item[(H)]
	$f\colon \Omega \times \R \times \R^N \to \R$ is a Carath\'{e}odory function such that
	\begin{enumerate}[leftmargin=0cm]
	    \item[(i)]
		There exists $\alpha \in \Lp{\frac{q_1}{q_1-1}}$ and $a_1, a_2\geq 0$ such that
		\begin{align}\label{growth_f}
		    |f(x,s,\xi)| & \leq a_1 |\xi|^{p \frac{q_1-1}{q_1}}+a_2|s|^{q_1-1}+\alpha(x)
		\end{align}
		for a.\,a.\,$x\in\Omega$, for all $s\in \R$ and for all $\xi\in\R^N$, where $1<q_1<p^*$ with the critical exponent $p^*$ given in \eqref{critical_exponent}.
	    \item[(ii)]
	      There exists $\omega \in \Lp{1}$ and $b_1,b_2\geq 0$ such that
	      \begin{align}\label{estimate_f}
		  f(x,s,\xi)s & \leq b_1|\xi|^p+b_2|s|^{p}+\omega(x)
	      \end{align}
	      for a.\,a.\,$x\in\Omega$, for all $s\in \R$ and for all $\xi\in\R^N$.
	      Moreover,
	      \begin{align}\label{condition_coeffizients}
		  b_1+b_2\lambda_{1,p}^{-1}< 1,
	      \end{align}
	      where $\lambda_{1,p}$ is the first eigenvalue of the Dirichlet eigenvalue problem for the $p$-Laplacian.
        \end{enumerate}
\end{enumerate}

\begin{example}
    The following function satisfies hypotheses (H), where for simplicity we drop the $x$-dependence,
    \begin{align*}
	f(s,\xi)=-d_1|s|^{q_1-2}s+d_2|\xi|^{p-1}\quad\text{for all } s\in\R \text{ and all }\xi \in\R^N,
    \end{align*}
    with $1<q_1<p^*$, $d_1 \geq 0$ and
    \begin{align*}
	0 \leq d_2< \frac{p}{p-1+\lambda_{1,p}^{-1}}.
    \end{align*}
\end{example}

We say that $u\in \Wpzero{\mathcal{H}}$ is a weak solution of problem \eqref{problem} if it satisfies
\begin{align}\label{weak_solution}
    \into \left(|\nabla u|^{p-2}\nabla u+\mu(x)|\nabla u|^{q-2}\nabla u \right)\cdot\nabla\ph \,dx=\into f(x,u,\nabla u)\ph \,dx
\end{align}
for all test functions $\ph \in \Wpzero{\mathcal{H}}$. Because of the embedding \eqref{embedding_sobolev} and the fact that $p<q$ along with \eqref{inequality_w1p} we easily see that a weak solution in \eqref{weak_solution} is well-defined.

Our main existence result reads as follows.

\begin{theorem}
    Let $1<p<q<N$ and let hypotheses \eqref{condition_poincare} and (H) be satisfied. Then problem \eqref{problem} admits at least one weak solution $u \in \Wpzero{\mathcal{H}}$.
\end{theorem}

\begin{proof}
    Let $\hat N_f\colon \Wpzero{\mathcal{H}} \subseteq \Lp{q_1}\to \Lp{q_1'}$ be the Nemytskij operator associated to $f$ and let $i^*\colon\Lp{q_1'}\to \Wpzero{\mathcal{H}}^*$ be the adjoint operator of the embedding $i\colon\Wpzero{\mathcal{H}}\to \Lp{q_1}$. For $u \in \Wpzero{\mathcal{H}}$ we define $N_f :=i^*\circ \hat N_f$ and set
    \begin{align}\label{operator}			
	\mathcal{A}(u)=A(u)-N_f(u).
    \end{align}

    From the growth condition on $f$, see \eqref{growth_f}, we easily know that $\mathcal{A}\colon \Wpzero{\mathcal{H}}\to \Wpzero{\mathcal{H}}^*$ maps bounded sets into bounded sets. Let us now prove that $\mathcal{A}$ is pseudomonotone, see Definition \ref{SplusPM}(b). To this end, let $\{u_n\}_{n\geq 1}\subseteq \Wpzero{\mathcal{H}}$ be a sequence such that
    \begin{align}\label{assumption_pseudomonotone}
	u_n \weak u\quad \text{in }\Wpzero{\mathcal{H}}\quad\text{and} \quad \limsup_{n\to \infty} \langle \mathcal{A}(u_n),u_n-u\rangle_{\mathcal{H}} \leq 0.
    \end{align}
    From the compact embedding in \eqref{embedding_sobolev} we obtain that
    \begin{align}\label{a1}
	u_n \to u\quad \text{in }\Lp{q_1}
    \end{align}
    since $q_1<p^*$. Using the strong convergence in $\Lp{q_1}$, see \eqref{a1}, along with H\"older's inequality and the growth condition on $f$ we obtain
    \begin{align*}
	\begin{split}
	    \lim_{n\to \infty} \into f(x,u_n,\nabla u_n) (u_n-u)\, dx&=0.
	\end{split}
    \end{align*}
    Therefore, we can pass to the limit in the weak formulation in \eqref{weak_solution} replacing $u$ by $u_n$ and $\ph$ by $u_n-u$.  This gives
    \begin{align}\label{a4}
	\limsup_{n\to\infty} \langle A(u_n),u_n-u\rangle_{\mathcal{H}}=\limsup_{n\to \infty} \langle \mathcal{A}(u_n),u_n-u\rangle_{\mathcal{H}}\leq 0.
    \end{align}
    From Proposition \ref{properties_operator_double_phase} we know that $A$ fulfills the $(\Ss_+)$-property and so we conclude, in view of \eqref{assumption_pseudomonotone} and \eqref{a4}, that $u_n\to u$ in $\Wpzero{\mathcal{H}}$. Hence, because of the continuity of $\mathcal{A}$, we have that $\mathcal{A}(u_n)\to \mathcal{A}(u)$ in $\Wpzero{\mathcal{H}}^*$ which proves that $\mathcal{A}$ is pseudomonotone.

    Next we show that the operator $\mathcal{A}$ is coercive, that is,
    \begin{align}\label{a5}
	\lim_{\|u\|_{1,\mathcal{H},0}\to \infty} \frac{\langle \mathcal{A}u, u\rangle_{\mathcal{H}}}{\|u\|_{1,\mathcal{H},0}}=+\infty.
    \end{align}
    From the representation of the first eigenvalue of the $p$-Laplacian, see \eqref{lambda_one}, replacing $r$ by $p$, we have the inequality
    \begin{align}\label{estimate_lambda}
	\|u\|^p_p \leq \lambda_{1,p}^{-1}
	\|\nabla u\|_p^p
	\quad\text{for all }u\in \Wpzero{p}.
    \end{align}
    Since $\Wpzero{\mathcal{H}}\subseteq \Wpzero{p}$ and by applying \eqref{estimate_lambda}, \eqref{estimate_f} and \eqref{inequality_lp} we derive
    \begin{align*}
	&\langle \mathcal{A}(u),u\rangle\\
	& = \into \left(|\nabla u|^{p-2}\nabla u+\mu(x)|\nabla u|^{q-2}\nabla u \right)\cdot\nabla u \,dx-\into f(x,u,\nabla u)u \,dx\\
	&\geq \|\nabla u\|_p^p+ \|u\|^q_{q,\mu} -b_1\|\nabla u\|_p^p-b_2\|u\|_p^p-\|\omega\|_1\\
	& \geq \left(1-b_1-b_2\lambda_{1,p}^{-1}\right)\|\nabla u\|_{p}^p+\|u\|^q_{q,\mu}-\|\omega\|_1\\
	& \geq \left(1-b_1-b_2\lambda_{1,p}^{-1}\right) \left(\|\nabla u\|_{p}^p+\|u\|^q_{q,\mu}\right)-\|\omega\|_1\\
	& \geq \left(1-b_1-b_2\lambda_{1,p}^{-1}\right) \min\left\{\|u\|_{1,\mathcal{H},0}^p, \|u\|_{1,\mathcal{H},0}^q\right\}-\|\omega\|_1.
    \end{align*}
    Therefore, since $1<p<q$ and \eqref{condition_coeffizients}, it follows \eqref{a5} and thus, the operator $\mathcal{A}\colon\Wpzero{\mathcal{H}}\to \Wpzero{\mathcal{H}}^*$ is coercive.

    Hence, the operator $\mathcal{A}\colon\Wpzero{\mathcal{H}}\to\Wpzero{\mathcal{H}}^*$ is bounded, pseudomonotone and coercive. Then Theorem \ref{theorem_pseudomonotone} provides $u\in \Wpzero{\mathcal{H}}$ such that $\mathcal{A}(u)=0$. By the definition of $\mathcal{A}$, see \eqref{operator}, the function $u$ turns out to be a weak solution of  problem \eqref{problem} which completes the proof.
\end{proof}

Let us now give sufficient conditions on the perturbation such that problem \eqref{problem} has a unique weak solution. To this end, we need the following stronger conditions on the convection term $f\colon\Omega\times\R\times\R^N\to\R$.
\begin{enumerate}[leftmargin=1cm]
    \item[(U1)]
	There exists $c_1\geq 0$ such that
	\begin{align*}
	    &(f(x,s,\xi)-f(x,t,\xi))(s-t) \leq c_1 |s-t|^2
	\end{align*}
	for a.\,a.\,$x\in\Omega$, for all $s,t \in\R$ and for all $\xi\in \R^N$.
    \item[(U2)]
	There exists $\rho\in \Lp{r'}$ with $1< r'<p^*$ and $c_2\geq 0$ such that $\xi\mapsto f(x,s,\xi)-\rho(x)$ is linear for a.\,a.\,$x\in\Omega$, for all $s\in \R$
	and
	\begin{align*}
	    |f(x,s,\xi)-\rho(x)|\leq c_2|\xi|
	\end{align*}
	for a.\,a.\,$x\in\Omega$, for all $s \in\R$ and for all $\xi\in \R^N$. Moreover,
	\begin{align}\label{condition_coeffizients2}
	    c_1\lambda_{1,2}^{-1}+c_2\lambda_{1,2}^{-\frac{1}{2}}<1,
	\end{align}
	where $\lambda_{1,2}$ is the first eigenvalue of the Dirichlet eigenvalue problem for the Laplace differential operator.
\end{enumerate}

\begin{example}
    The following function satisfies hypotheses (H), (U1) and (U2), where for simplicity we drop the $s$-dependence,
    \begin{align*}
	f(x,\xi)=\sum_{i=1}^N \beta_i \xi_i+\rho(x)\quad\text{for a.\,a.\,} x\in\Omega \text{ and for all }\xi \in\R^N,
    \end{align*}
    with $2=p\leq q_1<2^*$, $\rho \in \Lp{2}$ and
    \begin{align*}
	\|\beta \|_{\R^N}^2<\min\left\{1-\frac{1}{2}\lambda_{1,2}^{-1} \ , \ \lambda_{1,2}\right\}
    \end{align*}
    where $\beta=(\beta_1, \ldots, \beta_N)\in \R^N$.
\end{example}

Our uniqueness result reads as follows.

\begin{theorem}
    Let \eqref{condition_poincare}, (H), (U1), and (U2) be satisfied and let $2=p<q<N$.
    Then, problem \eqref{problem} admits a unique weak solution.
\end{theorem}

\begin{proof}
    Let $u,v\in \Wpzero{\mathcal{H}}$ be two weak solutions of \eqref{problem}. Taking in both weak formulations the test function $\ph=u-v$ and subtracting these equations result in
    \begin{align}\label{uniqueness_1}
      \begin{split}
	& \into |\nabla (u-v)|^2\,dx
	+ \into \mu(x) \left(|\nabla u|^{q-2}\nabla u-|\nabla v|^{q-2}\nabla u\right)\cdot \nabla (u-v)\,dx\\
	& = \into (f(x,u,\nabla u)-f(x,v,\nabla u))(u-v)\, dx\\
	&\quad + \into (f(x,v,\nabla u)-f(x,v,\nabla v))(u-v)\, dx.
	\end{split}
    \end{align}
    Since the second term on the left-hand side of \eqref{uniqueness_1} is nonnegative, we have the simple estimate
    \begin{align}\label{uniqueness_2}
      \begin{split}
	& \into |\nabla (u-v)|^2\,dx
	+ \into \mu(x) \left(|\nabla u|^{q-2}\nabla u-|\nabla v|^{q-2}\nabla u\right)\cdot \nabla (u-v)\,dx\\
	& \geq \into |\nabla (u-v)|^2 \,dx.
	\end{split}
    \end{align}
    The right-hand side of \eqref{uniqueness_1} can be estimated via (U1), (U2) and H\"older's inequality
    \begin{align}\label{uniqueness_3}
      \begin{split}
	& \into (f(x,u,\nabla u)-f(x,v,\nabla u))(u-v)\, dx\\
	&\quad + \into (f(x,v,\nabla u)-f(x,v,\nabla v))(u-v)\, dx\\
	& \leq  c_1\|u-v\|_2^2 +\into\left(f\left(x,v,\nabla \left(\frac{1}{2}(u-v)^2 \right)\right)-\rho(x) \right)\,dx\\
	&\leq c_1\|u-v\|_2^2+c_2\into |u-v| |\nabla (u-v)|\, dx\\
	& \leq\left(c_1\lambda_{1,2}^{-1}+c_2\lambda_{1,2}^{-\frac{1}{2}}\right)\|\nabla (u-v)\|_{2}^2.
	\end{split}
    \end{align}
    Combining \eqref{uniqueness_1}, \eqref{uniqueness_2} and \eqref{uniqueness_3} gives
    \begin{align}\label{uniqueness_4}
      \begin{split}
	& \|\nabla (u-v)\|_{2}^2=\into |\nabla (u-v)|^2 \,dx
	\leq \left(c_1\lambda_{1,2}^{-1}+c_2\lambda_{1,2}^{-\frac{1}{2}}\right)\|\nabla (u-v)\|_{2}^2.
	\end{split}
    \end{align}
    Then, by \eqref{condition_coeffizients2} and \eqref{uniqueness_4}, we get that $u=v$.
\end{proof}

\section*{Acknowledgment}

The second author thanks the Pedagogical University of Cracow for the kind hospitality during a research stay in April/May 2019.

The first author was supported by the National Science Center of Poland under Project No. 2015/19/B/ST1/01169,
and the H2020-MSCA-RISE-2018 Research and Innovation Staff Exchange Scheme Fellowship within the Project no. 823731CONMECH.

\end{document}